\documentclass[ps2pdf, reqno]{style}
\usepackage[T1]{fontenc}
\usepackage{times, pslatex, bbm, boxedminipage, mathtools, graphicx, paralist, cite, color, amsmath}
\RequirePackage{color}
\definecolor{e-mail}{rgb}{0,.40,.80}
\definecolor{reference}{rgb}{.20,.60,.22}
\definecolor{citation}{rgb}{0,.40,.80}
\usepackage[pdfstartview=FitH, colorlinks, linkcolor=reference, citecolor=citation, urlcolor=e-mail, backref ]{hyperref}

\newtheorem{thm}{Theorem}
\newtheorem{cor}[thm]{Corollary}
\newtheorem{lem}[thm]{Lemma}
\newtheorem{prop}[thm]{Proposition}
\theoremstyle{definition}
\newtheorem{defn}[thm]{Definition}
\theoremstyle{remark}

\numberwithin{thm}{section}
\theoremstyle{definition}
\numberwithin{equation}{section}

\title[Galois-theoretic proof of differential transcendence of incomplete Gamma function]{A Galois-theoretic proof of the differential transcendence \\ of the incomplete Gamma function}
\author{Carlos E. Arreche}
\email{carreche@gc.cuny.edu}
\address{Mathematics Department, The Graduate Center of the City University of New York, New York, NY 10016}
\thanks{This material is based upon work partially supported by a National Science Foundation (NSF) Graduate Research Fellowship (grant 40017-04-05) and NSF grant CCF-0952591.}

\begin{document}

\begin{abstract} We give simple necessary and sufficient conditions for the $\frac{\partial}{\partial t}$-transcendence of the solutions to a parameterized second order linear differential equation of the form \[\frac{\partial^2 Y}{\partial x^2}-p\frac{\partial Y}{\partial x}=0,\] where $p\in F(x)$ is a rational function in $x$ with coefficients in a $\frac{\partial}{\partial t}$-field $F$. This result is crucial for the development of an efficient algorithm to compute the parameterized Picard-Vessiot group of an arbitrary parameterized second-order linear differential equation over $F(x)$. Our criteria imply, in particular, the $\frac{\partial}{\partial t}$-transcendence of the incomplete Gamma function $\gamma(t,x)$, generalizing a result of \cite{johnson:1995}.
\end{abstract}

\maketitle


\section{Introduction} \label{intro}

The \emph{incomplete Gamma function} $\gamma(t,x)$ is defined by \begin{equation*} \gamma(t,x):=\int_0^xs^{t-1}e^{-s}ds  \end{equation*} for $\mathrm{Re}(t)>0$, and extended analytically to a multivalued meromorphic function on $\mathbb{C}\times\mathbb{C}$. It satisfies the second-order linear differential equation \begin{equation*}   \frac{\partial^2\gamma}{\partial x^2}-\frac{t-1-x}{x}\frac{\partial \gamma}{\partial x}=0. \end{equation*} The question arises whether $\gamma(t,x)$ satisfies any polynomial $\frac{\partial}{\partial t}$-differential equations (with coefficients in some differential field of interest). In \cite{johnson:1995}, the authors give two proofs of the following result, one analytic, and the other differential-algebraic (see \cite[Thm. 2]{johnson:1995}):\begin{thm} \label{original}The incomplete Gamma function $\gamma(t,x)$ is $\frac{\partial}{\partial t}$-transcendental over $\mathbb{C}(x,t)$.\end{thm} 

We will give a differential-algebraic proof of a stronger statement (Theorem~\ref{main}), as an application of the parameterized Picard-Vessiot theory developed in \cite{cassidy-singer:2006}. This differential Galois theory for parameterized linear differential equations is a generalization of the classical Picard-Vessiot theory \cite{kolchin:1948, vanderput-singer:2003}, and a special case of the theories presented in \cite{hardouin-singer:2008, landesman:2008}.

In \cite{hardouin-singer:2008}, the authors develop a Galois theory for (parameterized) difference equations, and apply it towards a novel proof \cite[Cor. 3.4.1]{hardouin-singer:2008} of H\"{o}lder's classical result on the $\frac{\partial}{\partial t}$-transcendence of the \emph{Gamma function} $\Gamma(t)$, on the basis that it satisfies the difference equation $\Gamma(t+1)=t\Gamma(t)$. Since the difference Galois group measures the algebraic dependencies among the derivatives of the solutions to this difference equation, the differential transcendence of $\Gamma(t)$ can be read off the difference Galois group (see \cite[\S3.1]{hardouin-singer:2008} for more details). We will follow an analogous strategy in our new proof of the differential transcendence of $\gamma(t,x)$.

Let us briefly describe the contents of the present work. In \S\ref{prelim}, we will review some terminology from differential algebra and summarize some results from the parameterized Picard-Vessiot theory \cite{cassidy-singer:2006} and the theory of linear differential algebraic groups \cite{cassidy:1972} that we will need to apply later on. In \S\ref{section:main} we will set the notation to be used for the rest of the paper, and deduce our main result (Theorem~\ref{main}) from Propositions~\ref{equiv-1} and~\ref{equiv-2}, which will be proved in \S\ref{section:proofs}. Theorem~\ref{main} states that if $\eta$ satisfies \begin{equation}\label{intro-eq} \delta^2\eta-p\delta\eta=0, \quad \delta\eta\neq 0;\end{equation} where $p\in K:=F(x)$, the $\{\delta,\partial\}$-field\footnote{The $\{\delta,\partial\}$-field structure of $K$ is defined by setting $\delta x=1$, $\partial x=0$, and $\delta|_{F}=0$ (see \S\ref{section:main}).} of rational functions in $x$ with coefficients in a $\partial$-closed\footnote{\label{note-1}These notions are defined in \S\ref{prelim}.} $\partial$-field $F$, then $\eta$ is $\partial$-transcendental\textsuperscript{\ref{note-1}} over $K$ if and only if none of the equations $\delta Y=\partial p$ and $\delta Y+pY=1$ admits a solution in $K$. In Corollary~\ref{original-cor} we drop the assumption that the $\delta$-constants\textsuperscript{\ref{note-1}} are $\partial$-closed, at the cost of obtaining only a sufficient criterion for the $\partial$-transcendence of $\eta$ over the ground field. Theorem~\ref{original} is proved as a straightforward consequence of Corollary~\ref{original-cor}.

The proof of Theorem~\ref{main} will be given in two steps. First, we prove in Proposition~\ref{equiv-1} that $\eta$ is $\partial$-transcendental if and only the parameterized Picard-Vessiot group ($\mathrm{PPV}$-group) corresponding to~\eqref{intro-eq} is ``large enough.'' Proposition~\ref{equiv-2} states that the largeness condition of Proposition~\ref{equiv-1} holds if and only if none of the equations $\delta Y=\partial p$ and $\delta Y+pY=1$ admits a solution in $K$.

Theorem~\ref{main} is a small (but crucial) part of a complete algorithm to compute the $\mathrm{PPV}$-group of a linear differential equation of the form \begin{equation}\label{general-eq-2}  \frac{\partial^2Y}{\partial x^2}+r_1\frac{\partial Y}{\partial x} +r_2Y=0, \end{equation} where $r_1,r_2\in K$. Most of this algorithm was developed in \cite{dreyfus:2011}, in the setting of several parametric derivations, but under the assumption that $r_1=0$. This restriction will be removed in a forthcoming paper (see \cite{arreche:2012} for a preliminary version), in the case of a single parametric derivation. See also \cite{min-ov-sing:2013}, where the authors describe how to compute higher dimensional $\mathrm{PPV}$-groups $\Gamma$ under the assumption that $\Gamma/R_u(\Gamma)$ is constant,\footnote{This is a generalization of the situation described by the equivalent conditions of Lemma~\ref{const-diag}.} where $R_u(\Gamma)$ denotes the unipotent radical of $\Gamma$ (see\cite[Def. 2.1 and \S2]{min-ov-sing:2013} for more details). In a different direction, the authors of \cite{gor-ov:2012} show how to check whether a system is isomonodromic\footnote{This is also defined as completely integrable system in \cite[Def. 3.8]{cassidy-singer:2006} (cf. the proof of Lemma~\ref{const-diag}).} by working with one parametric derivation at a time.

We have isolated these criteria of Theorem~\ref{main} from the rest of the algorithm \cite{arreche:2012,dreyfus:2011} to compute the $\mathrm{PPV}$-group of \eqref{general-eq-2} because of their independent interest and relative simplicity. Although the complete algorithm is somewhat involved, an effective test for differential transcendence such as Theorem~\ref{main} or Corollary \ref{original-cor} is already quite useful in algorithmic applications. Indeed, the main motivation of \cite{johnson:1995} is to decide when a system of algebraic differential equations can be extended to what they call an \emph{algebraic Mayer system}, i.e., a system of partial differential equations whose solutions are differentially algebraic with respect to each derivation (see \cite[\S1]{johnson:1995} for more details); they prove Theorem~\ref{original} as one of several counterexamples to show that this cannot always be done. The parameterized differential Galois theories presented in \cite{cassidy-singer:2006,hardouin-singer:2008,landesman:2008} (for example) provide a very natural setting for the study of such questions.


\section{Preliminaries}\label{prelim}

We refer to \cite{kaplansky:1976,vanderput-singer:2003} for more details concerning the following definitions. Every field considered in this work is assumed to be of characteristic zero. A field $K$ equipped with a finite set $\Delta:=\{\delta_1,\dots,\delta_m\}$ of pairwise commuting derivations (i.e., $\delta_i(ab)=a\delta_i (b)+ \delta_i( a)b$ and $\delta_i\delta_j=\delta_j\delta_i$ for each $a,b\in K$ and $1\leq i,j \leq m$) is called a $\Delta$\emph{-field}. We will often omit the parenthesis, and simply write $\delta a$ for $\delta(a)$. For $\Pi\subseteq\Delta$, we will denote the subfield of $\Pi$\emph{-constants} of $K$ by $K^\Pi:=\{a\in K \ | \ \delta a=0, \ \delta\in\Pi\}$. In case $\Pi=\{\delta\}$ is a singleton, we write $K^\delta$ instead of $K^\Pi$.

If $M$ is a $\Delta$-field and $K$ is a subfield such that $\delta(K)\subset K$ for each $\delta\in \Delta$, we say $K$ is a $\Delta$\emph{-subfield} of $M$ and $M$ is a $\Delta$\emph{-field extension} of $K$. If $y_1,\dots,y_n\in M$, we denote by \[K\langle y_1,\dots,y_n\rangle_\Delta\subseteq M\] the $\Delta$-subfield of $M$ generated over $K$ by all the derivatives of the $y_i$. We say that $y\in M$ is $\delta$\emph{-transcendental} over $K$ if the elements $y,\delta y,\delta^2y,\dots$ are algebraically independent over $K$. 

We say that $K$ is $\Delta$\emph{-closed} if every system of polynomial differential equations defined over $K$ which admits a solution in some $\Delta$-field extension of $K$ already has a solution in $K$. This last notion is discussed at length in \cite{kolchin:1974}. See \cite{cassidy-singer:2006} for a brief discussion, and more references.

We will not need to apply the parameterized Picard-Vessiot theory of \cite{cassidy-singer:2006} in its full generality, so let us briefly summarize the main facts that we will need. We work over a differential field $K$ equipped with a pair of commuting derivations $\Delta:=\{\delta,\partial\}$. We will sometimes refer to $\delta$ (resp., $\partial$) as the main (resp., parametric) derivation. Consider a linear differential equation with respect to the main derivation\begin{equation}\label{ppv-eq} \delta^nY+\sum_{i=0}^{n-1}r_i\delta^iY=0,\end{equation} where $r_i\in K$ for each $0\leq i\leq n-1$.

\begin{defn}  We say that a $\Delta$-field extension $M\supseteq K$ is a \emph{parameterized Picard-Vessiot extension} (or $\mathrm{PPV}$-extension) of $K$ for~\eqref{ppv-eq} if:
\begin{enumerate} 
\item There exist $n$ distinct, $K^\delta$-linearly independent elements $y_1,\dots, y_n\in M$ such that $\delta^ny_j+\sum_ir_i\delta^iy_j=0$ for each $1\leq j \leq n$.
\item $M=K\langle y_1,\dots y_n\rangle_\Delta$.
\item $M^\delta=K^\delta$.
\end{enumerate}

The \emph{parameterized Picard-Vessiot group} (or $\mathrm{PPV}$-group) is the group of $\Delta$-automorphisms of $M$ over $K$, and will be denoted by $\mathrm{Gal}_\Delta(M/K)$. The $K^\delta$-linear span of all the $y_j$ is the \emph{solution space}, and will be denoted by $\mathcal{S}$. \end{defn} 

 If $K^\delta$ is $\partial$-closed,\footnote{Although this assumption allows for a simpler exposition of the theory, several authors \cite{gill-gor-ov:2012, wibmer:2011} have shown that the parameterized Picard-Vessiot theory can be developed without assuming that $K^\delta$ is $\partial$-closed.} it is shown in \cite{cassidy-singer:2006} that a $\mathrm{PPV}$-extension and $\mathrm{PPV}$-group for \eqref{ppv-eq} over $K$ exist and are unique up to $K$-$\Delta$-isomorphism. The action of $\mathrm{Gal}_\Delta(M/K)$ is determined by its restriction to $\mathcal{S}$, which defines an embedding $\mathrm{Gal}_\Delta(M/K)\hookrightarrow\mathrm{GL}_n(K^\delta)$ after choosing a $K^\delta$-basis for $\mathcal{S}$. It is shown in \cite{cassidy-singer:2006} that this embedding identifies the $\mathrm{PPV}$-group with a linear differential algebraic group (Definition~\ref{ldag-def}), and from now on we will make this identification implicitly.

\begin{defn} \label{ldag-def}Let $F$ be a differentially closed $\partial$-field. We say that a subgroup $\Gamma\subseteq \mathrm{GL}_n(F)$ is a \emph{linear differential algebraic group} if $\Gamma$ is defined as a subset of $\mathrm{GL}_n(F)$ by the vanishing of a system of polynomial differential equations in the matrix entries, with coefficients in $F$. \end{defn}

The theory of linear differential  algebraic groups was pioneered in \cite{cassidy:1972} (see also \cite{kolchin:1984}). There is a parameterized Galois correspondence \cite[Thm. 3.5]{cassidy-singer:2006} between the linear differential algebraic subgroups $\Gamma$ of $\mathrm{Gal}_\Delta(M/K)$ and the intermediate $\Delta$-fields $K\subseteq L\subseteq M$, given by $\Gamma\mapsto M^\Gamma$ and $L\mapsto \mathrm{Gal}_\Delta(M/L)$. Under this correspondence, an intermediate $\Delta$-field $L$ is a $\mathrm{PPV}$-extension of $K$ (for some linear differential equation) if and only if $\mathrm{Gal}_\Delta(M/L)$ is normal in $\mathrm{Gal}_\Delta(M/K)$, and in this case the restriction map $\sigma\mapsto\sigma|_L:\mathrm{Gal}_\Delta(M/K)\twoheadrightarrow\mathrm{Gal}_\Delta(L/K)$ is surjective, with kernel given by $\mathrm{Gal}_\Delta(M/L)$.

The following classification theorems give many non-trivial examples of linear differential algebraic groups. We still assume that $F$ is a differentially closed $\partial$-field.

\begin{thm}[(Cassidy {\cite[Prop. 11]{cassidy:1972}})]\label{dag1}
Let $B$ be a differential algebraic subgroup of $\mathbb{G}_a(F )$, the additive group of $F$. Then, either $B=\mathbb{G}_a(F)$, or else there exists a unique nonzero monic operator $ \mathcal{D}\in F [\partial]$ such that \begin{equation*}B=\{b\in \mathbb{G}_a(K )\  | \  \mathcal{D}b=0\}.\end{equation*}
\end{thm}

\begin{thm}[(Cassidy {\cite[Prop. 31 and its Corollary]{cassidy:1972}})]\label{dag2}Let $A$ be a proper differential algebraic subgroup of $\mathbb{G}_m(F )$, the multiplicative group of $F$. Then, either $A=\nobreak\mu_n\subset F^\times$, the group of $n^{\text{th}}$ roots of unity for some $n\in\mathbb{N}$, or else there exists a unique nonzero monic operator $ \mathcal{D}\in K [\partial]$ such that \begin{equation*}A=\left\{a\in\mathbb{G}_m(F )\  \middle|\  \mathcal{D}\left(\tfrac{\partial a}{a}\right)=0\right\}.\end{equation*}
\end{thm}

We conclude this section by recalling the following classical result in the Picard-Vessiot theory. This result was originally proved by Ostrowski for fields of ``functions'', and generalized by Kolchin in \cite{kolchin:1968}.

\begin{thm}[(Kolchin-Ostrowski {\cite{kolchin:1968}})] \label{kolostro}Suppose that $E\subset \tilde{E}$ is a $\delta$-field extension such that $E^{\delta}=\tilde{E}^{\delta},$ and let $\smash{\{\mathfrak{f}_j\}_{j=0}^n}$ be a subset of $\tilde{E}$ such that $\smash{\delta\mathfrak{f}_j\in F}$ for each $j$. Then, there exists a nonzero polynomial $\Phi\in E[Y_0,\dots,Y_n]$ such that $\Phi(\mathfrak{f}_0,\dots,\mathfrak{f}_n)=0$ if and only if there exist elements $c_j\in E^\delta,$ not all zero, such that $\smash{\sum_{j=0}^nc_j\mathfrak{f}_j\in E}.$\end{thm}


\section{Main Result}\label{section:main}

We set once and for all the notation that we will use for the rest of the paper. Let $F$ be a differentially closed $\partial$-field of characteristic zero, and let $K:=F(x)$ with the structure of $\{\delta,\partial\}=:\Delta$-field defined by setting $\partial x=0$, $\delta x=1$, and $\delta|_{F}=0$. As in \S\ref{prelim}, $\delta$ is the main derivation and $\partial$ is the parametric derivation. Let $p\in K$, and consider the parameterized linear differential equation\vspace{-.1in} \begin{equation} \label{gen-eq-2} \delta^2Y-p\delta Y=0.\end{equation} Let $M$ be a $\mathrm{PPV}$-extension of $K$ for \eqref{gen-eq-2}, and let $\{1,\eta\}$ denote an $F$-basis for the solution space. Since $\mathrm{Gal}_\Delta(M/K)$ fixes the first basis vector in our chosen basis, we have that\footnote{In this paper, all differential Galois groups will act by linear transformations on the left.}\begin{equation}\label{full-ppv} \mathrm{Gal}_\Delta(M/K)\subseteq \left\{\begin{pmatrix}  1 & b \\ 0 & a  \end{pmatrix} \ \middle| \ a\in F^\times, \ b\in F\right\}.\end{equation} Note that the embedding $\mathrm{Gal}_\Delta(M/K)\hookrightarrow\mathrm{GL}_2(F)$ is given in this case by \[\sigma\mapsto\begin{pmatrix} 1 & b_\sigma \\ 0 & a_\sigma\end{pmatrix},\] where $\sigma(\eta)=a_\sigma\eta+b_\sigma$. Since $0\neq\delta\eta$ satisfies the parameterized first-order equation \begin{equation}\label{gen-eq-1} \delta Y-pY=0, \end{equation}the $\Delta$-subfield $L:=K\langle\delta\eta\rangle_\Delta\subseteq M$ is a $\mathrm{PPV}$-extension of $K$ for \eqref{gen-eq-1}. Since $\sigma(\delta\eta)=a_\sigma\delta\eta$, we have that $\mathrm{Gal}_\Delta(M/L)$ is given by $\{\sigma\in\mathrm{Gal}_\Delta(M/K) \ | \ a_\sigma=1\}$, which implies that \begin{equation}\label{rad-ppv} \mathrm{Gal}_\Delta(M/L)\subseteq \left\{\begin{pmatrix}  1 & b \\ 0 & 1  \end{pmatrix} \ \middle|  \ b\in F\right\}\simeq\mathbb{G}_a(F).\end{equation} From now on, we will identify $\mathrm{Gal}_\Delta(M/K)$ and $\mathrm{Gal}_\Delta(M/L)$ with their images in $\mathrm{GL}_2(F)$, as in \eqref{full-ppv} and \eqref{rad-ppv}.

\begin{lem}\label{finite-type} With notation as above, the $\Delta$-field $L:=K\langle\delta\eta\rangle_\Delta$ is finitely generated over $K$.
\end{lem}

\begin{proof} It is enough to show that $\delta\eta\in L$ is not $\partial$-transcendental over $K$. By \cite[Prop. 3.3]{singer:2011}, $\mathrm{Gal}_\Delta(L/K)$ is a proper subgroup of $\mathbb{G}_m(F)$. By Theorem~\ref{dag2}, there exists a non-zero operator $\mathcal{D}\in F[\partial]$ such that \[\mathcal{D}\left(\frac{\partial a_\sigma}{a_\sigma}\right)=0\] for every $\sigma\in\mathrm{Gal}_\Delta(L/K)$, where $\sigma:\delta\eta\mapsto a_\sigma\delta\eta$. We claim that \begin{equation}\label{lemma-diff-alg} \mathcal{D}\left(\frac{\partial(\delta\eta)}{\delta\eta}\right)\in K.\end{equation} By the Galois correspondence \cite[Thm. 3.5]{cassidy-singer:2006}, it suffices to show that every $\sigma\in\mathrm{Gal}_\Delta(L/K)$ fixes this element. To see this, note that\vspace{-.1in} \[ \sigma\left(\mathcal{D}\left( \frac{\partial(\delta\eta)}{\delta\eta}\right)\right) =\mathcal{D}\left(\frac{\partial(\delta\eta)}{\delta\eta}+\frac{\partial a_\sigma}{a_\sigma}\right)= \mathcal{D}\left(\frac{\partial(\delta\eta)}{\delta\eta}\right). \vspace{.05in}\] This concludes the proof of the Lemma.
\end{proof}

\begin{thm}\label{main} Let $p\in K$, and suppose that $\eta$ satisfies $\delta^2\eta=p\delta\eta$ and $\delta\eta\neq 0$. Then, $\eta$ is $\partial$-transcendental over $K$ if and only if the following conditions hold: \begin{enumerate} 
\item \label{condition1} The equation $\delta Y=\partial p$ does not admit a solution in $K$.
\item \label{condition2} The equation $\delta Y+pY=1$ does not admit a solution in $K$.
\end{enumerate} \end{thm}

\begin{proof}This is a consequence of Propositions~\ref{equiv-1} and \ref{equiv-2}, which will be proved in \S\ref{section:proofs}. Proposition~\ref{equiv-1} states that $\eta$ is $\partial$-transcendental over $L$ if and only if $\mathrm{Gal}_\Delta(M/L)=\mathbb{G}_a(F)$, whereas Proposition~\ref{equiv-2} states that $\mathrm{Gal}_\Delta(M/L)=\mathbb{G}_a(F)$ if and only if conditions~\eqref{condition1} and \eqref{condition2} hold. The $\partial$-transcendence of $\eta$ over $L$ implies that $\eta$ is also $\partial$-transcendental over $K$. On the other hand, Lemma~\ref{finite-type} implies that if $\eta$ is $\partial$-transcendental over $K$, then it must also be $\partial$-transcendental over $L$. \end{proof}

The following corollary shows that Theorem~\ref{main} can be used to establish differential transcendence over $\Delta$-fields whose field of $\delta$-constants is not necessarily $\partial$-closed.

\begin{cor} \label{original-cor}Let $S:=R(x)$ be a $\Delta:=\{\delta,\partial\}$ field with $\delta x=1$, $\partial x=0$, and $S^\delta=R$, and let $\bar{R}$ denote an algebraic closure of $R$. Let $S\subset T$ be a $\Delta$-field extension, and suppose that $\eta\in T$ satisfies $\delta^2\eta=p\delta\eta$ and $\delta\eta\neq 0$. If none of the equations $\delta Y=\partial p$ and $\delta Y+pY=1$ admits a solution in $\bar{R}(x)$, then $\eta$ is $\partial$-transcendental over $S$. \end{cor}

\begin{proof} 
Let $R'$ denote a $\partial$-differential closure\footnote{See \cite{kolchin:1974}, where the term \emph{constrained closure} is used instead.} of $T^\delta$. Then $T\otimes_{T^\delta}R'$ (resp. $S\otimes_RR'$) is a domain \cite[\S8.2]{gill-gor-ov:2012}, and we denote by $T'$ (resp. $S'$) its field of fractions. Since the embedding $S\hookrightarrow T'$ factors through $S\hookrightarrow S'\hookrightarrow T'$, it suffices to show that $\eta\in T'$ is $\partial$-transcendental over $S'$. By Theorem~\ref{main} applied to $K=S'$, it is enough to show that none of the equations $\delta Y=\partial p$ and $\delta Y+pY=1$ admits a solution in $S'$. 

The natural map $S\hookrightarrow S'$ induces an embedding $R\hookrightarrow R'$, which may be extended to an $R$-embedding $\bar{R}\hookrightarrow R'$ because $R'$ is algebraically closed. By assumption, the equations $\delta Y=\partial p$ and $\delta Y+pY=1$ do not admit solutions in $\bar{R}(x)$, and therefore they do not admit solutions in $S'=R'(x)$, either. This follows from the explicit methods presented in \cite{eremenko:1998} or in \cite[\S3]{singer:1981} for the construction of rational solutions to (first-order) linear differential equations with coefficients in a field of rational functions: when such rational solutions exist, one can write down a system of algebraic equations over $R$ in the unknown coefficients of the sought-for rational function. If the $R$-variety defined by this system of equations does not have an $\bar{R}$-point, it cannot have an $R'$-point, either. \end{proof}

We conclude this section by deducing Theorem~\ref{original} from Corollary~\ref{original-cor}.

\begin{proof}[Proof of {Thm.~\ref{original}}] We apply Corollary~\ref{original-cor} with $R:=\mathbb{C}(t)$, $S:=R(x)=\mathbb{C}(t,x)$, $\Delta:=\{\frac{\partial}{\partial x},\frac{\partial}{\partial t}\}$ and $\eta:=\gamma(t,x)$. By Corollary~\ref{original-cor}, the $\frac{\partial}{\partial t}$-transcendence of $\gamma(t,x)$ over $S$ will follow from the nonexistence of solutions in $\overline{\mathbb{C}(t)}(x)$ to any of the equations $\frac{\partial Y}{\partial x} =  \frac{1}{x}$ and $\frac{\partial Y}{\partial x}+\frac{t-1-x}{x}Y=1$, where $\overline{\mathbb{C}(t)}$ denotes the algebraic closure of $\mathbb{C}(t)$. It is clear that the first equation does not admit rational solutions (as all of its solutions are of the form $\log (x) + c$ for some $c\in\overline{\mathbb{C}(t)}$). We proceed by contradiction: suppose there exists $r\in\overline{\mathbb{C}(t)}(x)$ such that \begin{equation}\label{gamma-contra}\frac{\partial r}{\partial x}+\frac{t-1-x}{x}r=1.\end{equation} First note that $\partial r/\partial x\neq 0$, whence $r$ must have a pole somewhere on $\smash{\mathbb{P}^1(\overline{\mathbb{C}(t)})}$. But $r$ can only have poles at $\{0,\infty\}$, because otherwise the left-hand side of \eqref{gamma-contra} will have poles. If $r$ had a pole at $0$, the residue of $\frac{t-1-x}{x}$ at $0$ would have to be an integer, which is clearly false. Hence, $r$ can only have a pole at $\infty$, that is, $r$ must be a polynomial in $x$. Moreover, $r$ must be divisible by $x$, since otherwise the left-hand side of \eqref{gamma-contra} would have a pole at $0$. But then the degree of the polynomial on the left-hand side of \eqref{gamma-contra} is equal to the degree of $r$, which is at least $1$. This contradiction concludes the proof.\end{proof}


\section{Proofs}\label{section:proofs}

We keep the same notation as in \S\ref{section:main}: $K=F(x)$, $F$ is $\partial$-closed, $L$ (resp. $M$) is a $\mathrm{PPV}$ extension of $K$ for \eqref{gen-eq-1} (resp. \eqref{gen-eq-2}), and $\eta\in M$ satisfies $\delta^2\eta=p\delta\eta$ and $\delta\eta\neq 0$. We begin by showing that $\eta$ is $\partial$-transcendental over $L$ if and only if $\mathrm{Gal}_\Delta(M/L)$ is as large as possible. 

\begin{prop}\label{equiv-1}We have that $\mathrm{Gal}_\Delta(M/L)=\mathbb{G}_a(F)$ if and only if $\eta$ is $\partial$-transcendental over $L$. \end{prop}

\begin{proof}For the first implication, suppose that there exists a polynomial $\Phi\in L[Y_0, Y_1,\dots,Y_n]$, not identically zero, such that $\Phi(\eta,\partial\eta,\dots,\partial^n\eta)=0$. Since \[\delta(\partial^j\eta)=\partial^j(\delta\eta)\in L,\] we apply Theorem~\ref{kolostro} with $E:=L$, $\tilde{E}:=M$, and $\mathfrak{f}_j:=\partial^{j}\eta$ to conclude that there exist $c_0,\dots,c_n\in F$, not all zero, such that $\sum_{j=0}^nc_j\partial^j\eta\in L$. Since \begin{equation*}\sigma\left(\sum_{j=0}^nc_j\partial^j\eta\right)=\sum_{j=0}^nc_j\partial^j(\eta+b_\sigma)=\sum_{j=0}^nc_j\partial^j\eta + \sum_{j=0}^nc_j\partial^jb_\sigma\end{equation*} for every $\sigma\in\mathrm{Gal}_\Delta(M/L)$, we have that $\sum_{j=0}^nc_j\partial^jb_\sigma=0$ for all $\sigma$, which implies that $\mathrm{Gal}_\Delta(M/L)$ is a proper subgroup of $\mathbb{G}_a(F)$.

For the opposite implication\footnote{This implication is proved in \cite[Example 7.2]{cassidy-singer:2006}; we follow their proof.}, assume that $\mathrm{Gal}_\Delta(M/L)$ is a proper subgroup of $\mathbb{G}_a(F)$. By Theorem~\ref{dag1}, there exists a non-zero differential operator $\sum_{j=0}^nc_j\partial^j$ such that $\sum_{j=0}^nc_j\partial^jb_\sigma=0$ for every $\sigma\in\mathrm{Gal}_\Delta(M/L)$. Thus, we have that \vspace{-.125in}\[\sigma\bigl(\sum_{j=0}^nc_j\partial^j\eta\bigr)=\sum_{j=0}^nc_j\partial^j\eta,\] for all $\sigma\in\mathrm{Gal}_\Delta(M/L)$. By the parameterized Galois correspondence \cite[Thm. 3.5]{cassidy-singer:2006}, we have that $\sum_{j=0}^nc_j\partial^j\eta\in L$. Hence, $\eta$ is not $\partial$-transcendental over $L$.
\end{proof}

The following two Lemmas relate conditions \eqref{condition1} and \eqref{condition2} of Theorem~\ref{main} to certain properties of the linear differential algebraic group $\mathrm{Gal}_\Delta(M/K)$.

\begin{lem} \label{const-diag} The equation $\delta Y=\partial p$ admits a solution in $K$ if and only if $\partial a_\sigma=0$ for every $\sigma\in\mathrm{Gal}_\Delta(L/K)$.
\end{lem}

\begin{proof}This is a special case of \cite[Prop. 3.9]{cassidy-singer:2006}, taking into account that the existence of $q\in K$ such that $\delta q=\partial p$ coincides with the integrability conditions \cite[Def. 3.8]{cassidy-singer:2006} for the system \eqref{gen-eq-1}.\end{proof}

\begin{lem} \label{trivial-radical} The equation  $\delta Y+pY=1$ admits a solution in $K$ if and only if $\mathrm{Gal}_\Delta(M/L)=0$.\end{lem}

\begin{proof} If $r\in K$ satisfies $\delta r+pr=1$, we have that \[\delta (r\delta\eta)=(\delta r+pr)\delta\eta=\delta\eta.\] Therefore $\delta(\eta -r\delta\eta)=0$, which implies that $\eta\in L$ because $M^\delta=L^\delta=F$.

For the opposite implication, suppose that $\mathrm{Gal}_\Delta(M/L)=0$. Then $\sigma\left(\frac{\delta\eta}{\eta}\right)=\frac{\delta\eta}{\eta}$ for every $\sigma\in\mathrm{Gal}_\Delta(M/K)$ (cf. the discussion in \S\ref{section:main}), and therefore there exists $r\in K$ such that $r\delta\eta=\eta$. Applying $\delta$ on both sides of the latter equation, and using the fact that $\delta^2\eta=p\delta\eta$ and $\delta\eta\neq0$, we obtain that $\delta r +pr=1$.\end{proof}

\begin{prop}\label{equiv-2} Conditions \eqref{condition1} and \eqref{condition2} of Theorem~\ref{main} are satisfied if and only if $\mathrm{Gal}_\Delta(M/L)=\mathbb{G}_a(F)$.\end{prop}

\begin{proof} We begin by proving the necessity of the conditions. We recall that condition~\eqref{condition1} (resp. condition~\eqref{condition2}) of Theorem~\ref{main} states that the equation $\delta Y=\partial p$ (resp. $\delta Y+pY=1$) does not admit a solution in $K$. By Lemma~\ref{trivial-radical}, if condition~\eqref{condition2} fails, then $\mathrm{Gal}_\Delta(M/L)=0$. Let us show that, if condition~\eqref{condition1} fails, then $\mathrm{Gal}_\Delta(M/L)\neq \mathbb{G}_a(F)$. By Lemma~\ref{const-diag}, the existence of $q\in K$ such that $\delta q=\partial p$ implies that $\mathrm{Gal}_\Delta(L/K)\subseteq\mathbb{G}_m(F^\partial)$, which implies that $\mathrm{Gal}_\Delta(L/K)$ is either finite or equal to $\mathbb{G}_m(F^\partial)$. If $\mathrm{Gal}_\Delta(L/K)$ is finite, \cite[Lem. 3.2]{singer:2011} shows that $\mathrm{Gal}_\Delta(M/L)$ is a proper subgroup of $\mathbb{G}_a(F)$. In \cite[pp. 159--160]{singer:2011}, it is shown that the linear differential algebraic group\[ \left\{\begin{pmatrix} 1 & b \\ 0 & a\end{pmatrix} \ \middle| \ a,b\in F, \ a\neq 0, \ \partial a=0 \right\}   \] cannot be a $\mathrm{PPV}$-group over $K$, as an application of \cite[Thm. 1.1]{singer:2011}. Therefore, if $\mathrm{Gal}_\Delta(L/K)=\mathbb{G}_m(F^\partial)$, we have that $\mathrm{Gal}_\Delta(M/L)$ is a proper subgroup of $\mathbb{G}_a(F)$.

Let us prove the sufficiency of the conditions. By \cite[Lem. 3.6(2)]{hardouin-singer:2008}, if $\mathrm{Gal}_\Delta(M/L)$ were a nontrivial proper subgroup of $\mathbb{G}_a(F)$, we would have that $\mathrm{Gal}_\Delta(L/K)\subseteq\mathbb{G}_m(F^\partial)$. By Lemma~\ref{const-diag}, this conclusion is equivalent to the failure of condition~\eqref{condition1}. Therefore, condition~\eqref{condition1} implies that $\mathrm{Gal}_\Delta(M/L)$ is either $0$ or all of $\mathbb{G}_a(F)$. By Lemma~\ref{trivial-radical}, condition~\eqref{condition2} implies that $\mathrm{Gal}_\Delta(M/L)$ is not $0$, whence conditions \eqref{condition1} and \eqref{condition2} together imply that $\mathrm{Gal}_\Delta(M/L)=\mathbb{G}_a(F)$.\end{proof}


\begin{acknowledgements} Professor Alexey Ovchinnikov patiently read several drafts of this work and made numerous suggestions for improvements and corrections. Professor Sergey Gorchinskiy kindly pointed out to me several mistakes in an earlier version of this paper, and made several suggestions leading to the correct formulations of Theorem~\ref{main} and Corollary~\ref{original-cor}. It is a pleasure to acknowledge their essential contributions to the present form of this work. Any errors or inaccuracies that might remain are my own.
\end{acknowledgements}


\bibliography{gammabib}{} \bibliographystyle{spmpsci} \nocite{*}

\end{document}